\documentclass[11pt,draft,reqno]{amsart}
\usepackage{amsmath,amsfonts,amssymb}
\usepackage{mathrsfs}

\def\R{\mathbb R}

\def\Z{\mathbb Z}

\def\N{\mathbb N}

\def\1{\mathbf 1}

\def\1{\bold 1}

\def\le{\leqslant}

\def\ge{\geqslant}

\def\sign{\mathrm{sign}\,}

\def\caP{\mathcal{P}}

\theoremstyle{theorem}
\newtheorem{theorem}{Theorem}[section]
\newtheorem{proposition}[theorem]{Proposition}
\newtheorem{lemma}[theorem]{Lemma}

\newtheorem{remark}[theorem]{Remark}
\newtheorem{corollary}[theorem]{Corollary}
\newtheorem{definition}[theorem]{Definition}

\numberwithin{equation}{section}

\theoremstyle{plain}
\newtoks\thehProclaim
\newtheorem*{Proclaim}{\the\thehProclaim}

\usepackage[height = 23.5cm, a4paper, hmargin={2.5cm , 2.5cm}]{geometry}

\begin{document}

\title[Regularity of the solution of the Prandtl equation]
{Regularity of the solution \\ of the Prandtl equation}

\author{V.~E.~Petrov, T.~A.~Suslina}

\address{TWELL Ltd., Shvetsova str. 12, St.~Petersburg, 198095, Russia}
\email{petrov$\_$twell@list.ru}

\address{St.~Petersburg State University
\\
Universitetskaya nab. 7/9
\\
St.~Petersburg, 199034, Russia}

\email{t.suslina@spbu.ru}

\keywords{Prandtl equation, weak solution, Fourier integral transformation, integral transformation on the interval}

\thanks{Mathematics Subject Classification (2010): MSC 47G20}

\thanks{Supported by Russian Science Foundation (project 17-11-01069).}

\begin{abstract}
Solvability and regularity of the solution of  the Dirichlet problem for the Prandtl equation 
$$
{u(x)\over p(x)}-   {1\over 2\pi}\int_{-1}^1 {u'(t) \over t-x} \,dt = f(x)
$$
is studied. It is assumed that $p(x)$ is a positive function on $(-1,1)$ such that $\sup \frac{(1-x^2)}{ p(x)} < \infty$. We introduce the scale of spaces  $\widetilde{H}^s(-1,1)$ in terms of the special integral transformation on the interval $(-1,1)$. We obtain theorem about existence and uniqueness  of the solution in  the classes $\widetilde{H}^{s}(-1,1)$ with $0\le s \le 1$. In particular, for $s=1$ the result is as follows: if $r^{1/2} f \in L_2$, then 
$r^{-1/2} u, r^{1/2} u' \in L_2$, where $r(x)=1-x^2$.  
\end{abstract}

\maketitle

\section*{Introduction}
\setcounter{section}{0}
\setcounter{equation}{0}

\subsection{The Prandtl equation: physical motivation} 
The Prandtl equation  
\begin{equation}
{u(x)}-  p(x) {1 \over 2 \pi}\intop_{-1}^1
{u'(t) \over t-x} \,dt = p(x)f(x),\quad  u(- 1)= u(1) =0, \label{0.10}
\end{equation} 
 is one of the universal equations of mathematical physics. It is used in almost all cases  when thin plates are studied.  
 In aerodynamics and hydrodynamics \cite{G,S} this equation describes the circulation (the load averaged along the chord $c(x)$) on a three-dimensional thin wing with the span $L$ in a stream running at the angle of attack 
 $\alpha_0(x)$ with the speed $U_0$. In this case,  $p(x)=a_0c(x)/L$ and $f(x)=\alpha_0(x)U_0,$ where $a_0$ 
 is some constant coefficient.  
 In magnetostatics \cite {Kr}, equation  \eqref{0.10} describes the surface (averaged along the thickness $\delta(x)$) magnetization induced in a thin plate of width $L$ of a ferromagnetic material with susceptibility $\varkappa$ 
 by a transverse external magnetic field with a tangential component $f(x)$.
 In this case, $p(x)=\varkappa\delta(x)/L$. In two-dimensional mechanics \cite{Kal} equation \eqref{0.10} is the main tool  in the study of contact problems and calculations of stiffeners.  Equation  \eqref{0.10} can be considered as the potential theory equation for the following boundary value problem  
\begin{equation}\left\{\begin{aligned}
&\Delta \Phi(x,y)=0,\quad (x,\,y)\in\R^2\backslash[-1,1],\\
& \partial_y \Phi^+(x,0)= \partial_y \Phi^-(x,0),\quad x\in[-1,1],\\
&\Phi^+(x,0)-\Phi^-(x,0)+p(x) \partial_y \Phi^+(x,0) =p(x)f(x),\quad x\in[-1,1],
\label{0.20}
\end{aligned}\right.\end{equation}
if the solution is found as a double layer potential with density $u(t)$:
$$
\Phi(x,y)={1\over2\pi}\intop_{-1}^1 u(t)\left.{\partial\over\partial \tau}\ln{1\over \rho}\right|_{\tau=0}\,dt,\quad \rho=\sqrt{(x-t)^2+(y-\tau)^2},\qquad{}
$$
or for the adjoint problem (in the sense of the Cauchy--Riemann conditions)
\begin{equation}\left\{\begin{aligned}
&\Delta \Psi(x,y)=0,\quad (x,\,y)\in\R^2\backslash[-1,1],\\
&\Psi^+(x,0)=\Psi^-(x,0),\quad x\in[-1,1],\\
&\partial_y \Psi^+(x,0) - \partial_y \Psi^-(x,0) - \partial_x \left( p(x) \partial_x \Psi(x,0)\right)=\partial_x(p(x)f(x)),\quad x\in[-1,1],\end{aligned}\right.
\label{0.30}
\end{equation}
if the solution is found as a single layer potential with density $u'(t)$:
$$
\Psi(x,y)={1\over2\pi}\intop_{-1}^1 u'(t)\,\ln{1\over \rho}\,dt,\quad \rho=\sqrt{(x-t)^2+y^2}.\qquad{}
$$
Boundary value problems of the type \eqref{0.20}, \eqref{0.30} arise from general three-dimensional problems, when one of the parameters of the domain (or of the surface) becomes ``thin''. 

As a rule, $p(x)\ge0,$ this function can vanish only at the ends of the interval. In general, the potential theory equations for boundary value problems \eqref{0.20}, \eqref{0.30} contain hypersingular operators. In problem \eqref{0.20} this is the normal derivative of the double layer potential, and in problem \eqref{0.30} this is the second tangential derivative of the single layer potential. However, the degeneracy at the ends of the coefficient $p(x)$ facing the integral operator smooths out such a singularity. In particular, if $p(x)=(1-x^2)p_0(x)$ and 
$p_0(\pm1)\ne 0,\,\infty,$ then the operator that corresponds to the second term in  \eqref{0.10}  is simply singular. 

Usually it is not possible to solve the Prandtl equation exactly\footnote{The only widely known case  is the elliptic wing. If $p(x)=p_0\sqrt{1-x^2}$ and $f(x)=1,$ then $u(x)={p_0\over p_0+4}\sqrt{1-x^2}.$}. Therefore, the main literature  concerning applications is devoted to the search for convenient numerical schemes. One of the most popular schemes is  the Multopp method. A significant part of the monograph \cite{Kal} is devoted to the justification of this method. However, the conditions under which the Multopp method is justified are too restrictive (for example,   the function ${\sqrt{1-x^2}\over p(x)}$ is required to be H{\"o}lder).

Thus, a rigorous functional study of the Prandtl equation under natural (from the point of view of physical applications) conditions on the coefficient $p(x)$ is very relevant. 

\subsection{Relation to the theory of the Schr{\"o}dinger equation with fractional Laplacian}
The integral operator in the Prandtl equation \eqref{0.10} can also be represented as 
\begin{equation*}
-{1\over 2\pi}\intop_{-1}^1{u'(t) \over t-x} \,dt =  \sqrt{-{d^2\over dx^2}} \,[u](x).
\end{equation*}
The number of papers devoted to the study of equations with fractional  Laplacian is huge. In particular, generalized (or weak) solutions of boundary value problems for  equations of the form 
$$
(-\Delta)^\sigma u(x) +V(x)u(x)=f(x),\quad x \in \Omega\subset \R^n,\quad 0< \sigma <1,
$$
are constructed.
A survey, some new results and extensive literature on this topic can be found in  \cite{DGV} and \cite{F}. 
Note that the one-dimensional case  is usually not distinguished, although it has an important feature: the boundary of the domain  in the one-dimensional case is disconnected  (it consists of two points~$\pm1$).

\subsection{Main results}
We study problem  \eqref{0.10} assuming that the coefficient  $p(x)$ is a  positive function on $(-1,1)$,
 it may vanish for $x=\pm 1$,  but the order of zeros  is not higher than the first degree (see condition  \eqref{p}
 on $V(x)=p(x)^{-1}$ below). This case includes interesting  from the practical point of view examples of triangular wing of an airplane and a composite wing  with a chord break of first kind.

In Section \ref{sec1}, using the Fourier transform approach, we define the weak solution of problem  
\eqref{0.10}  in the class $H^{1/2}_{00}(-1,1)$ (consisting of functions $u$ on the interval $(-1,1)$ that can be extended by zero to functions of the Sobolev class $H^{1/2}(\R)$). The equation is replaced by 
  appropriate integral identity, and it is assumed that $f(x)$ belongs to the class dual to  $H^{1/2}_{00}(-1,1)$ (with respect to the pairing in $L_2(-1,1)$). 
We prove  theorem about existence and uniqueness of the solution.  
(Of cource, analogs of these results are known; see  \cite{DGV}.)

However, the obtained solution from the class $H_{00}^{1/2}(-1,\,1)$ does not even have to be continuous. 
And from the physical point of view, one is interested in continuous solutions. 
Therefore, we need to investigate  the question of additional regularity of the solution (under the previous assumptions on the coefficient $p(x)$ by  strengthening the assumptions on the function $f(x)$).

In the authors opinion, the Fourier transformation on the axis is not the most convenient tool for studying the problem  on a finite interval. Therefore, investigating regularity of the solution, we use a special integral transformation $\caP$ on the interval, which was introduced and studied  in \cite{Petrov_PMA}. Earlier, application of the transform $\caP$ made it possible to  solve many problems on the  interval, which were  previously either  solved in a much more laborious way  \cite{Petrov_DRAN},  or remained open \cite{AP, Petrov_AA}. 
The definition and the main properties of  $\caP$ are described in Section~\ref{sec2}.
We introduce the interpolational scale of spaces 
$\widetilde{H}^{s}(-1,1)$, $s \ge 0$, in terms of the transform $\caP$.

In Section~\ref{sec3},  we give an independent definition  of the weak solution 
$u \in \widetilde{H}^{1/2}(-1,1)= H_{00}^{1/2}(-1,1)$ of problem \eqref{0.10} in terms of the transform $\caP$, 
and then establish  theorem about additional regularity of  the solution. Namely, 
under the assumption that 
$$\intop_{-1}^1 (1-x^2) |f(x)|^2 \, dx < \infty,$$
 it is proved that the solution belongs to the class
$\widetilde{H}^1(-1,1)$ (which is distinguished by condition \eqref{0033}) and satisfies the estimate
$$
\intop_{-1}^1 \left( \frac{|u(x)|^2}{1-x^2} + (1-x^2) |u'(x)|^2\right) \, dx
\le C  \intop_{-1}^1 (1-x^2) |f(x)|^2 \, dx.
$$
By interpolation, we obtain solvability of problem  \eqref{0.10} in $\widetilde{H}^{s}(-1,1)$ with \hbox{$1/2 \le s \le 1$} (under the assumption that  $f$ belongs to the class dual to $\widetilde{H}^{1-s}(-1,1)$).
For \hbox{$s> 1/2$}, the solution is continuous on the closed interval $[-1,1]$ and satisfies the boundary conditions $u(\pm 1)=0$. Finally, by duality arguments, we prove solvability of problem \eqref{0.10} in $\widetilde{H}^{s}(-1,1)$, where \hbox{$0 \le s < 1/2$} (under the assumption  that $f$ belongs to the class dual to $\widetilde{H}^{1-s}(-1,1)$). In particular, for $s=0$ we obtain existence and uniqueness of the so called
``very weak'' solution from the class $\widetilde{H}^0(-1,1)$
(which is distinguished by the condition $\int^1_{-1} (1-x^2)^{-1} |u(x)|^2\, dx < \infty$).

Thus, the application of an adequate apparatus allowed us to introduce the suitable scale of spaces and prove additional regularity  of the solution  under wide assumptions on the coefficient $p(x)$ and the function $f(x)$ on the right-hand side.

\subsection{Acknowledgements} The authors are grateful to A.~I.~Nazarov and F.~V.~Petrov for useful discussions.

\section{The weak solution of the Prandtl equation\label{sec1}}

\subsection{Statement of the problem}
It is convenient to put $V(x) := p(x)^{-1}$ and rewrite 
 the Prandtl equation \eqref{0.10} with the Dirichlet conditions as 
\begin{equation}
V(x) u(x) -   {1 \over 2 \pi}\intop_{-1}^1
{u'(t) \over t-x} \,dt = f(x),\quad  u(- 1)= u(1) =0.\label{2.1.10}
\end{equation} 
Here and below,  singular integrals are understood  in the mean value sense.
We assume that $V(x)$ is a measurable function on $(-1,1)$ satisfying the following conditions:
\begin{equation}
 V(x) \ge 0\ \text{for a.~e.}\ x\in(-1,\,1);\quad (1-x^2) V(x) \le M<\infty.\label{p}
\end{equation}
The assumptions on the right-hand side $f$ will be formulated later. 

Our first goal is to define  the weak solution of equation \eqref{2.1.10} and to prove theorem about existence and uniqueness of the solution.

\subsection{Definition of the weak solution. The approach via the Fourier transform}

As usual in the theory of weak solutions of boundary value problems, we will first conduct formal considerations that will  tell us how to define the solution correctly.   
Due to boundary conditions, it is natural to consider the function $u(x)$ on $\R$, extending it by zero.  
Let $g(x)$ be a function on $\R$ supported on $[-1,\,1]$. Multiply equation \eqref{2.1.10} by $\overline{g(x)}$ and integrate:
\begin{equation}
\intop_{-1}^1 V(x)u(x)\,\overline{g(x)} \,dx +{1\over 2 \pi} \intop_{-\infty}^{\infty}  \overline{g(x)} 
\intop_{-\infty}^{\infty} {u'(t) \over x-t} \,dt\,dx=\intop_{-1}^1 f(x)\,\overline{g(x)}\,dx .\label{2.1.20}
\end{equation}
Next, we use the Fourier transform, which is taken in the form
$$
\widehat{u}(\zeta)=\intop_{-\infty}^{\infty}  u(x)\,e^{i x\zeta}\,dx,\quad u(x)={1\over2\pi}\intop_{-\infty}^{\infty}   \widehat{u}(\zeta) e^{-i x\zeta}\,d\zeta.
$$
Let us transform the second term in the left-hand side of \eqref{2.1.20} with the help of 
   the Parseval identity
\begin{equation}
\label{Pars_F}
 \intop_{-\infty}^{\infty} v(x) \, \overline{g(x)}\,dx = \frac{1}{2\pi} \intop_{-\infty}^{\infty}   \widehat{v}(\zeta)\,
\overline{\widehat{g}(\zeta)}\, d\zeta.
\end{equation}
 Observe  that the internal integral in \eqref{2.1.20} is the Fourier-convolution\footnote{In what follows, we will also  consider  convolution for another integral transform.}  $u'(t)* t^{-1}.$ The Fourier-image of $u'(t)$ is given by  
  $-i\zeta\,\widehat{u}(\zeta)$,  and the Fourier-image of the function $t^{-1}$ equals $i\pi\,\sign\zeta.$
 Thus, equation \eqref{2.1.20} can be written as 
\begin{equation}
\intop_{-1}^1 V(x) u(x)\,\overline{g(x)} \,dx+ \frac{1}{4\pi} \intop_{-\infty}^{\infty}  |\zeta|\, \widehat{u}(\zeta)\,\overline{\widehat{g}(\zeta)}\,d\zeta=\intop_{-1}^1 f(x)\,\overline{g(x)}\,dx .\label{2.1.30}
\end{equation}
Denote the sesquilinear form in the left-hand side of \eqref{2.1.30} by
\begin{equation}
[u, g]:= \intop_{-1}^1 V(x)u(x)\,\overline{g(x)} \,dx+ \frac{1}{4\pi} \intop_{-\infty}^{\infty}  |\zeta|\, \widehat{u}(\zeta)\,\overline{\widehat{g}(\zeta)}\,d\zeta.\label{2.1.30a}
\end{equation}

A natural class to look for the weak solution  is the class\footnote{In the literature, various notations are used for this class and its dual; we accept the notation from \cite{LM}.}
$H_{00}^{1/2} := H_{00}^{1/2}(-1,1)$ defined as a subspace in the Sobolev space $H^{1/2}(\R)$ consisting 
of functions  equal to zero almost everywhere  outside the segment $[-1,1]$.
Concerning the properties of this class of functions, see \cite[Chapter 1, Section 11.5]{LM}.
The norm in $H_{00}^{1/2}$ is defined as the standard norm in $H^{1/2}(\R)$:
$$
\| u\|^2_{H_{00}^{1/2}} = \| u \|^2_{H^{1/2}(\R)} = 
\frac{1}{2\pi} \intop_{-\infty}^{\infty}  (1+ |\zeta|^2)^{1/2} |\widehat{u}(\zeta)|^2\,d\zeta,\quad u\in H_{00}^{1/2}.
$$
Note that the set $C_0^\infty(-1,1)$ is dense in  $H_{00}^{1/2}$.

We need the following property of the functions from $H_{00}^{1/2}$ 
(for the sake of completeness, we provide the proof).

\begin{lemma}
Any function $u \in H_{00}^{1/2}(-1,1)$ satisfies 
\begin{equation}
\intop_{-1}^1 {\ |u(x)|^2\over 1-x^2}\,dx\le \frac{1}{2} \intop_{-\infty}^\infty |\zeta| |\widehat{u}(\zeta)|^2\,d\zeta. 
\label{0005}
\end{equation}
\end{lemma}

\begin{proof}
We extend $u \in H_{00}^{1/2}$ by zero, keeping the same notation.  
Let us use the following identity valid on the class $H^{1/2}(\R)$:
\begin{equation}
\intop_{-\infty}^\infty  \intop_{-\infty}^\infty \frac{|u(x) - u(y)|^2}{|x-y|^2}\, dx\,dy = 
\intop_{-\infty}^\infty |\zeta| |\widehat{u}(\zeta)|^2\,d\zeta. 
\label{double-int}
\end{equation}
To check \eqref{double-int}, denote the left-hand side by $J[u]$ and substitute $y=x+z$:
\begin{equation*}
\begin{aligned}
J[u] &= 
\intop_{-\infty}^\infty \frac{dz}{|z|^2} \intop_{-\infty}^\infty |u(x+z) - u(x)|^2\, dx 
= \frac{1}{2\pi} \intop_{-\infty}^\infty \frac{dz}{|z|^2} \intop_{-\infty}^\infty |e^{-i z \zeta}-1|^2 |\widehat{u}(\zeta)|^2\, d\zeta =
\\ 
&= \frac{1}{2\pi}\intop_{-\infty}^\infty  |\widehat{u}(\zeta)|^2\, d\zeta   \intop_{-\infty}^\infty \frac{|e^{-i z \zeta}-1|^2}{|z|^2}\, dz.
\label{double-int-2}
\end{aligned}
\end{equation*}
We have used the Parseval identity \eqref{Pars_F} and the Fubini theorem. 
The internal integral is equal to  
$2\pi |\zeta|$. This implies \eqref{double-int}.

Since $u(y)=0$ for $y\in\R\backslash[-1,1],$ from \eqref{double-int} it follows that 
\begin{equation*}
\intop_{-\infty}^\infty |\zeta| |\widehat{u}(\zeta)|^2\,d\zeta \ge
\intop_{-1}^1 |u(x) |^2 dx \intop_{-\infty}^{-1} \frac{dy}{|x-y|^2} +
\intop_{-1}^1 |u(x)|^2 dx \intop_{1}^{\infty} \frac{dy }{|x-y|^2} =  2\intop_{-1}^1 {\ |u(x)|^2\over 1-x^2}\,dx,
\end{equation*}
which proves \eqref{0005}.
\end{proof}

Let us check that the relation $u \in H_{00}^{1/2}$ is equivalent to the condition $[u,u]< \infty$.

Suppose that $[u,u]< \infty$.  Then, by \eqref{0005}, 
\begin{equation}
\intop_{-1}^1  |u(x)|^2 \, dx \le \intop_{-1}^1 {\ |u(x)|^2\over 1-x^2}\,dx\le \frac{1}{2} \intop_{-\infty}^\infty |\zeta| |\widehat{u}(\zeta)|^2\,d\zeta. 
\label{01}
\end{equation}
 From \eqref{2.1.30a} and \eqref{01}, using the Parceval identity \eqref{Pars_F}, we obtain
 \begin{equation}
 \| u\|^2_{H_{00}^{1/2}} \le \intop_{-1}^1  |u(x)|^2 \, dx +\frac{1}{2\pi} \intop_{-\infty}^{\infty}   |\zeta| |\widehat{u}(\zeta)|^2\,d\zeta
\le (2\pi + 2) [u,u].
\label{02}
\end{equation}

On the other hand, by \eqref{p} and \eqref{0005}, 
$$
\intop_{-1}^1  V(x) |u(x)|^2 \,dx \le M \intop_{-1}^1  \frac{|u(x)|^2}{1- x^2}\,dx
\le  M \pi \|u\|^2_{H^{1/2}_{00}}.
$$
  Together with \eqref{2.1.30a}, this implies that 
\begin{equation}
 [u,u] \le (M \pi+1/2) \| u\|^2_{H_{00}^{1/2}}, \quad u\in H_{00}^{1/2}.
\label{03b}
\end{equation}

 From \eqref{02} and \eqref{03b} it follows that the form $[u,u]^{1/2}$ determines the norm in $ H_{00}^{1/2}$
 equivalent to the standard one. Then the sesquilinear form $[u,g]$ given by \eqref{2.1.30a}
 can be taken as  the inner product in this space.

Let us write identity \eqref{2.1.30} in the form 
\begin{equation}
[u,\,g]=(f,\,g),
\label{2.1.45}
\end{equation}
where $(f,\,g) := (f,\,g)_{L_2(-1,1)}$.
The natural class for  $f$ is the space $(H_{00}^{1/2})^*$ dual to 
$H_{00}^{1/2}$ with respect to the pairing in $L_2(-1,1)$. 
In other words, a distribution $f$ that is an anti-linear continuous functional over $C_0^\infty(-1,1)$
belongs to $(H_{00}^{1/2})^*$ if 
 \begin{equation}
\label{03}
\sup_{0 \ne u \in C_{0}^{\infty}(-1,1)} \frac{|(f,u)|}{\|u\|_{H_{00}^{1/2}}} < \infty.
\end{equation}
Then the pairing $(f,u)$ in $L_2(-1,1)$ extends to the pairs $f\in (H_{00}^{1/2})^*$ and $u \in 
H_{00}^{1/2}$, and the left-hand side of \eqref{03} is taken as the norm of $f$ in $(H_{00}^{1/2})^*$. Note that  
\begin{equation}
\label{03a}
\|f \|_{(H_{00}^{1/2})^*} =
\sup_{0 \ne u \in H_{00}^{1/2}} \frac{|(f,u)|}{\|u\|_{H_{00}^{1/2}}}.
\end{equation}

  Now, we give a definition of the weak solution of problem \eqref{2.1.10}. 

\begin{definition}\label{def1.1}
Let $f \in (H_{00}^{1/2})^*$. An element $u\in H_{00}^{1/2}$ satisfying integral identity \eqref{2.1.30} for any $g\in H_{00}^{1/2}$ is called the weak solution of problem \eqref{2.1.10}.
\end{definition}

\begin{theorem}\label{th1.2}
Suppose that  $V(x)$ satisfies \eqref{p}. Then for any 
$f \in (H_{00}^{1/2})^*$ there exists a unique weak solution $u\in H_{00}^{1/2}$ of problem  \eqref{2.1.10}.
The solution satisfies the following estimate{\rm :}
\begin{equation}
\label{05}
\|u \|_{H_{00}^{1/2}} \le (2\pi +2) \|f\|_{(H_{00}^{1/2})^*}.
\end{equation}
\end{theorem}

\begin{proof}
We write identity \eqref{2.1.30} in the form \eqref{2.1.45}.
By \eqref{03a}, the right-hand side $l_f(g)=(f,g)$ is an anti-linear continuous functional over 
$g \in  H_{00}^{1/2}$, and 
\begin{equation}
\label{04}
| l_f(g)| \le \|f\|_{(H_{00}^{1/2})^*}  \|g\|_{H_{00}^{1/2}}.
\end{equation}

We consider $H_{00}^{1/2}$ as the Hilbert space with the inner product $[u,g]$. 
By the Riesz theorem about the general form of an anti-linear continuous functional in a Hilbert space,  there exists a unique element $u\in H_{00}^{1/2}$ such that $l_f(g)=[u,\,g],\ \forall g\in H_{00}^{1/2}$, i.~e., identity \eqref{2.1.45} holds. This proves existence and uniqueness of the solution.

Estimate \eqref{05} follows from the identity $[u,u] = l_f(u)$ and relations \eqref{02} and \eqref{04}.
\end{proof}

\textbf{Example}.
Denote $r(x):= 1 - x^2$ and assume that $f \in L_{2,r}(-1,1)=: L_{2,r},$ i.~e., 
\begin{equation}
\|f\|^2_{L_{2,r}} = \intop_{-1}^1 (1-x^2)\,|f(x)|^2\,dx <\infty.\label{H}
\end{equation}
Then, by \eqref{0005},  
\begin{equation}
\bigl|(f,\,g)\bigr| \!\le\! \left(\,\intop_{-1}^1 r(x)\,|f(x)|^2 \,dx\right)^{\!1/2}\left(\,\intop_{-1}^1{|g(x)|^2\over r(x)}\,dx\right)^{\!1/2}\! \le\! \sqrt{\pi} \|f\|_{L_{2,r}} \|g \|_{H^{1/2}_{00}},\quad g \in H_{00}^{1/2}.\label{2.1.47}
\end{equation}
It follows that $L_{2,r} \subset (H_{00}^{1/2})^*$ and 
$
\|f \|_{(H_{00}^{1/2})^*} \le \sqrt{\pi} \|f\|_{L_{2,r}}.
$
Together with  \eqref{05}, this yields the following estimate of the solution:
\begin{equation}\label{012}
\|u \|_{H_{00}^{1/2}}  \le (2\pi+2)\sqrt{\pi}   \|f\|_{L_{2,r}}. 
\end{equation}

Note that, in general, functions of class $H_{00}^{1/2}$ may be discontinuous.  
However, under condition $f \in L_{2,r}$ we can expect that the solution is more regular.
It turns out that it is inconvenient to study this question via the Fourier transform approach. We will study this problem using  another integral transform.

\section{The integral transformation $\caP$\label{sec2}}
We need the integral transform $\caP$ on the interval which was studied in details in \cite{Petrov_PMA, Petrov_DRAN}. In the present section, main properties of this transform 
are described. In terms of $\caP$, we introduce the scale of the Hilbert spaces $\widetilde{H}^s(-1,1)$; this material is new.

\subsection{Definition of the transform $\caP$} Consider the space $\widetilde{L}_2(-1,\,1)$ consisting of all measurable functions on the interval $(-1,\,1)$ such that
\begin{equation*}
\|u\|^2_{\widetilde{L}_2(-1,1)} := \intop_{-1}^1 {\ |u(x)|^2\over1-x^2}\,dx <\infty.
\end{equation*}
For functions $u \in \widetilde{L}_2(-1,\,1)$, the transform $\caP$ is defined as follows\footnote{In \cite{Petrov_PMA, Petrov_DRAN}, $\caP$ was defined as a transform from the interval to the imaginary axis.}:
\begin{equation}\left\{\begin{aligned}
U(\xi):=&\caP[u](\xi)=\intop_{-1}^1 u(y)\left({1-y\over1+y}\right)^{i\xi}\,\frac{dy}{1-y^2},\quad \xi\in\R,
\\
u(x)=&\caP^{-1}[U](x)={1\over\pi}\intop_{-\infty}^\infty  U(\xi)\,\left({1-x\over1+x}\right)^{-i\xi}\,d\xi,\quad x\in (-1,\,1).
\label{0020}
\end{aligned}\right.\end{equation}
Here $U \in L_2(\R).$
The exact explanation of the meaning of relations \eqref{0020} can be found in \cite{Petrov_PMA, Petrov_DRAN}.
In what follows, by default, the originals are denoted by lowercase letters, and the $\caP$-images by the corresponding  uppercase letters.

\subsection{The Parceval identity}  We have the following Parceval identity for the transform $\caP$: 
\begin{equation}
\intop_{-1}^1 u(y)\,\overline{g(y)}\, \frac{dy}{1-y^2}= \frac{1}{\pi} 
\intop_{-\infty}^\infty U(\xi)\,\overline{G(\xi)}\,d\xi.\label{0050}
\end{equation}
The relation $u\in \widetilde{L}_2(-1,\,1)$ is equivalent to the relation $U \in L_2(\R).$
Thus, the operator $\frac{1}{\sqrt{\pi}}\caP$ is a unitary mapping of the space $\widetilde{L}_2(-1,\,1)$
onto $L_2(\R)$.

\subsection{Relationship between $\caP$ and the Fourier transform\label{sec2.3}} 
The transform $\caP$ is related to the Fourier transform by the following change of variables:  
\begin{equation}\label{zamena1}
x=\tanh\omega;\quad \omega={1\over2}\ln{1-x\over1+x},\ \omega \in \R;\quad u(x)=u_1(\omega).
\end{equation}
Then the $\caP$-image of a function $u(x)$ and the Fourier-image of $u_1(\omega)$ satisfy 
\begin{equation}\label{zamena2}
U(\xi)={\widehat{u}}_1(2\xi).
\end{equation}
It is easily seen that the linear mapping $A: \widetilde{L}_2(-1,1) \to L_2(\R)$, defined by the rule 
$(Au)(\omega)= u_1(\omega)$, is an isometric isomorphism: 
$$
\| u \|_{\widetilde{L}_2(-1,1)} = \| Au\|_{L_2(\R)}.
$$

\subsection{The $\caP$-transformation of the derivatives. The spaces $\widetilde{H}^n$, $n \in {\mathbb Z}_+$} 
Now, we define the space  $\widetilde{H}^{n}(-1,1)=: \widetilde{H}^{n}$ of measurable functions $u(x)$ having the generalized derivatives up to order $n$ on the interval $(-1,\,1)$ and satisfying  
\begin{equation*}
|\!|\!| u |\!|\!|_{\widetilde{H}^{n}}^2 := \intop_{-1}^1 \sum_{m=0}^n \Bigl|(1-x^2)^m u^{(m)}(x)\Bigr|^2\,
\frac{dx}{1-x^2}<\infty.
\end{equation*}
Then $\widetilde{H}^{0}(-1,\,1)= \widetilde{L}_2(-1,\,1).$  Obviously,  $\widetilde{H}^{n}(-1,1) \subset  H^n_{\mathrm{loc}}(-1,\,1).$ 

If $u\in \widetilde{H}^{n},$ then, integrating by parts, we obtain 
\begin{equation}
\caP\left[\Bigl((1-y^2){d\over dy}\Bigr)^nu(y)\right](\xi)=(2i\xi)^n U(\xi).\label{0030}
\end{equation}

Let us explain this in details for  $n=1$. The condition $u \in \widetilde{H}^1$ means that  
\begin{equation}
  |\!|\!| u  |\!|\!|_{\widetilde{H}^{1}}^2 = \intop_{-1}^1 \left({\ |u(x)|^2\over1-x^2}+(1-x^2)|u'(x)|^2\right)dx
<\infty. \label{0033}
\end{equation}
 Since  $\widetilde{H}^1 \subset H^1_{\text{loc}}(-1,1)$, then, by the Sobolev embedding theorem, any function \hbox{$u \in \widetilde{H}^1$} is absolutely continuous inside the interval $(-1,1)$. We will show that it is continuous on the closed interval $[-1,1]$. Indeed, for any  $x,y \in (-1,1)$ we have 
$$
\bigl| |u(y)|^2- |u(x)|^2 \bigr|=\left|2 \operatorname{Re} \intop_x^y u(t)\,\overline{u'(t)}\,dt\right|\le \intop_x^y\left( {|u(t)|^2\over1-t^2}+
 (1-t^2)|u'(t)|^2\right)\,dt,
$$
and from condition \eqref{0033} it follows that $\bigl| |u(y)|^2- |u(x)|^2 \bigr| \to 0$, as $x\to 1$ and $y\to 1$. 
Then, applying the Cauchy criterion, we conclude that there exists a finite limit 
$
\lim_{y \to 1-0} |u(y)|^2.
$
 This limit is equal to zero, since otherwise the integral $\int_{-1}^1 \frac{|u(t)|^2}{1-t^2}\,dt$ will be divergent.
 Hence, $u(y)$ converges to zero, as  $y \to 1-0$,  
 and we can put  
 $$
 u(1):= \lim_{y \to 1-0} u(y) =0.
 $$ 
 In a similar fashion, we check that there exists a limit  $u(-1):= \lim_{y \to -1+0} u(y)  =0$.
 Using that $u$ is continuous inside the interval  $(-1,1)$, we obtain $u \in C[-1,1]$.

Integrating by parts and using the boundary conditions $u(-1)=u(1)=0$, we obtain relation \eqref{0030} with $n=1$:
 \begin{equation}
\caP\left[(1-y^2)u'(y)\right](\xi)
=
\intop_{-1}^1 u'(y)\left({1-y\over1+y}\right)^{i\xi}\,dy=2i\xi\, U(\xi),\quad u \in \widetilde{H}^1. \label{0040}
\end{equation}

Next,  from \eqref{0050}, \eqref{0033}, and \eqref{0040} it follows that 
\begin{equation*}
  |\!|\!| u  |\!|\!|_{\widetilde{H}^{1}}^2 
= \frac{1}{\pi}\intop_{-\infty}^\infty \left( 1 + 4 \xi^2\right) |U(\xi)|^2 \, d\xi, \quad u \in \widetilde{H}^1. \label{**1}
\end{equation*}

Relation \eqref{0030} for  
$u \in \widetilde{H}^n$ with arbitrary $n \in \N$ is proved similarly. It turns out that  $|\!|\!| u |\!|\!|^2_{\widetilde{H}^n}$ admits  two-sided estimates by  
\begin{equation}
\|u\|^2_{\widetilde{H}^n} := \frac{1}{\pi}\intop_{-\infty}^\infty \left( 1 + 4 \xi^2\right)^n |U(\xi)|^2 \, d\xi.
 \label{**1aa}
\end{equation}

\begin{remark}\label{rem_Hn}
The above arguments  allow us to give  another definition of the spaces  $\widetilde{H}^n$ \emph{(}for any $n \in \Z_+$\emph{)} in terms of the transform $\caP$\emph{:}
$\widetilde{H}^n(-1,1)$  is the class of functions $u \in \widetilde{L}_2(-1,1)$ for which the norm 
$\|u\|_{\widetilde{H}^n}$ given by \eqref{**1aa} is finite.
By \eqref{zamena2},  it is obvious that the mapping $A$ \emph{(}see Subsection \emph{\ref{sec2.3})}
restricted to $\widetilde{H}^n(-1,1)$ is an isometric isomorphism of the space $\widetilde{H}^n(-1,1)$ onto the Sobolev space $H^n(\R)$\emph{:}
$$
\| u \|_{\widetilde{H}^n} = \| Au\|_{H^n(\R)}, \quad u \in \widetilde{H}^n(-1,1).
$$
\end{remark}

\subsection{The spaces $\widetilde{H}^s$\label{sec2.5}} Now, we introduce the space $\widetilde{H}^s(-1,1)=: \widetilde{H}^s$ 
with arbitrary index $s \ge 0$, as a subspace of $\widetilde{L}_2(-1,1)= \widetilde{H}^0$, consisting of functions  
$u$ such that  
\begin{equation*}
\|u\|^2_{\widetilde{H}^s} := \frac{1}{\pi}\intop_{-\infty}^\infty \left( 1 + 4 \xi^2\right)^s |U(\xi)|^2 \, d\xi < \infty.
\end{equation*}
 Automatically, for $n \in \Z_+$ this agrees with  the previous definition (see Remark \ref{rem_Hn}).
For any $s \ge 0$,  the mapping $A$ restricted to 
$\widetilde{H}^s(-1,1)$ is an isometric isomorphism of the space $\widetilde{H}^s(-1,1)$ onto the Sobolev space $H^s(\R)$\emph{:}
$$
\| u \|_{\widetilde{H}^s} = \| Au\|_{H^s(\R)}, \quad u \in \widetilde{H}^s(-1,1).
$$
Since $H^s(\R)$, $s \ge 0$, forms an interpolational scale of Hilbert spaces, the same is true for the spaces $\widetilde{H}^s(-1,1)$, $s \ge 0$.

\begin{remark}
Recall that in the space $H^s(\R)$ with $s \ne [s] =:k$ the following norm given in the internal terms is equivalent to the standard norm\emph{:}
$$
\boldsymbol{|} u_1 \boldsymbol{|}_{H^s(\R)}^2 = \sum_{j=0}^k 
\intop_{-\infty}^\infty |u^{(j)}_1(\omega)|^2\, d\omega + \intop_{-\infty}^\infty\intop_{-\infty}^\infty
\frac{|u_1^{(k)}(\omega) - u_1^{(k)}(\tau)|^2}{|\omega - \tau|^{1+2\{s\}}}
\, d\omega\, d\tau,\quad \{s\} = s -k.
$$  
Using the isomorphism $A$, we see that the following norm given in the internal terms is equivalent to the 
 standard norm in the space $\widetilde{H}^s(-1,1)$ with \hbox{$s \ne [s] =k$}\emph{:}
$$
\begin{aligned}
\boldsymbol{|} u \boldsymbol{|}_{\widetilde{H}^s}^2 &= \sum_{j=0}^k 
\intop_{-1}^1 \bigl|(1-x^2)^j u^{(j)}(x)\bigr|^2\, \frac{dx}{1-x^2}
\\
&+ \intop_{-1}^1\intop_{-1}^1
\frac{\left|(1-x^2)^k u^{(k)}(x) - (1-y^2)^k u^{(k)}(y)\right|^2}{\left| \operatorname{ln} \Bigl(\frac{1-x}{1+x}\Bigr) -  \operatorname{ln} \Bigl(\frac{1-y}{1+y}\Bigr) \right|^{1+2\{s\}}}
\, \frac{dx}{1-x^2}\, \frac{dy}{1-y^2}.
\end{aligned}
$$  
However, below we will not use this norm.
\end{remark}

Using the isomorphism $A$, from the density of $C_0^\infty(\R)$ in $H^s(\R)$,
we deduce the following statement.

\begin{proposition}
\label{prop_plotno}
 For any $s \ge 0$, the set $C_0^\infty(-1,1)$ is dense in the space $\widetilde{H}^s(-1,1)$. 
\end{proposition}

We need the following statement, which is an analog of the Sobolev embedding theorem.

\begin{proposition}\label{prop_2.1}
 Let $s>1/2$. Then we have the continuous embedding  
$$
\widetilde{H}^s(-1,1) \subset C[-1,1].
$$
Any function $u \in \widetilde{H}^s(-1,1)$ satisfies the boundary conditions  
\begin{equation}
\label{gran_usl}
u(-1)= u(1)=0
\end{equation} 
and the estimate 
 \begin{equation}
 \label{vlozh}
 \| u \|_{C[-1,1]} \le
C(s) \| u \|_{\widetilde{H}^{s}}, \quad C(s) = \left( \frac{\Gamma(s-1/2)}{2 \sqrt{\pi} \,\Gamma(s)} \right)^{1/2}.
\end{equation}
\end{proposition}

\begin{proof}
We rely on the relation between  $\caP$ and the Fourier transform (see Subsection \ref{sec2.3}).

 By the Sobolev embedding theorem,  for $s>1/2$ the space $H^s(\R)$ is continuously embedded in 
 the space of uniformly continuous functions. Next, a function $u_1(\omega)$ from the class $H^s(\R)$ satisfies  
 $\widehat{u}_1\in L_1(\R)$, because 
\begin{equation}
\label{Hs_2}
 \begin{aligned}
 \intop_{-\infty}^\infty |\widehat{u}_1(\xi)|\, d\xi 
 \le \!
 \left(\,\intop_{-\infty}^\infty (1+\xi^2)^{-s}\, d\xi\right)^{1/2}
 \left(\,\intop_{-\infty}^\infty (1+\xi^2)^s|\widehat{u}_1(\xi)|^2 \, d\xi \right)^{1/2} 
 \! = C_s \|u_1\|_{H^s(\R)} < \infty.
 \end{aligned}
\end{equation}
 Here $C_s= \left(\frac{2 \pi^{3/2}\, \Gamma(s-1/2)}{\Gamma(s)}\right)^{1/2}$. Then from the inversion formula 
 $$
 u_1(\omega) = \frac{1}{2\pi} \int_{-\infty}^\infty e^{- i \xi \omega} \widehat{u}_1(\xi)\, d\xi,
 $$ 
by the Riemann--Lebesgue lemma, it follows that 
\begin{equation}
\label{Hs_3}
\lim_{\omega \to \pm \infty} u_1(\omega) =0.
\end{equation}
Thus, $u_1(\omega)$ belongs to the class $C_0(\R)$ of uniformly continuous functions on $\R$ satisfying conditions \eqref{Hs_3}.
By \eqref{Hs_2} and the inversion formula, 
\begin{equation}
\label{Hs_4}
\max_{\omega \in \R} |u_1(\omega)| \le \frac{1}{2\pi}\int_{-\infty}^\infty |\widehat{u}_1(\xi)|\, d\xi 
 \le \frac{1}{2\pi} C_s \|u_1\|_{H^s(\R)}.
\end{equation}

Now, let $u\in \widetilde{H}^s(-1,1)$, $s>1/2$. Making substitutions \eqref{zamena1}, we see that 
the function $(Au)(\omega)=u_1(\omega)$ belongs to  $H^s(\R)$, and $\|u_1\|_{H^s(\R)} = \|u\|_{\widetilde{H}^s}$. 
Then from the properties of the function $u_1(\omega)$ proved above  it follows that $u(x)$ is uniformly continuous on the closed interval $[-1,1]$ and satisfies the boundary conditions \eqref{gran_usl}. 
The identity  
$$\|u\|_{C[-1,1]} = \max_{x \in [-1,1]} |u(x)|= \max_{\omega \in \R} |u_1(\omega)|$$
and inequality \eqref{Hs_4} imply estimate \eqref{vlozh}.
\end{proof}

\subsection{Transformation of the distributions} The transform $\caP$ for distributions is defined by duality, similarly to definition of the Fourier transform for distributions. 
Consider the space $\mathcal X$ consisting of the functions $\psi \in C^\infty[-1,1]$ such that 
the following seminorms 
$$
\sup_{x \in [-1,1]} (1-x^2)^n|\psi^{(n)}(x)| \left( 1 + \operatorname{ln}^2 \Bigl(\frac{1-x}{1+x}\Bigr)\right)^{k}
$$
are finite for all $n,k \in {\mathbb Z}_+$. Convergence in $\mathcal X$ is understood as convergence with respect to this set of seminorms. Let ${\mathcal X}'$ be the class of distributions dual to  $\mathcal X$ with respect to the pairing in  $\widetilde{L}_2(-1,1)$. The transform $\caP$ takes the class $\mathcal X$ onto the Schwartz class 
${\mathcal S}(\R)$. Let $u \in {\mathcal X}'$. Then  $\caP u \in {\mathcal S}'(\R)$ is defined by 
$$
(\caP u, \varphi)_{L_2(\R)} = (u, \caP^* \varphi)_{\widetilde{L}_2(-1,1)}= \pi (u, \caP^{-1}\varphi)_{\widetilde{L}_2(-1,1)},
\quad \varphi \in {\mathcal S}(\R).
$$

Below we need the result of calculation of the $\caP$-image of the distribution $v(y)={1\over y}$ (understood in the mean value sense); see \cite[(1.9)]{Petrov_PMA}:
\begin{equation}
\caP\left[{1\over y}\right](\xi)=\intop_{-1}^1 {1\over y}\left({1-y\over1+y}\right)^{i\xi}\,{dy\over1-y^2}=-i\pi\coth\pi\xi,\label{0090}
\end{equation}
where the integral is understood in the mean value sense both at the point $y=0,$ and at the ends of the interval.

\subsection{The convolution formulas} The following convolution formulas are valid for the transform $\caP$:
\begin{align}
\caP\left[\;\intop_{-1}^1 u(x)\,v\left({y-x\over1-xy}\right)\, \frac{dx}{1-x^2}\right](\xi)&=U(\xi)\,V(\xi), \label{0060}\\
\caP\left[\;\intop_{-1}^1 u'(x)\,v\left({y-x\over1-xy}\right)\,dx\right](\xi)&=2i\xi\,U(\xi)\,V(\xi).\label{0070}
\end{align}
In \eqref{0060} it is assumed that $u\in \widetilde{L}_2(-1,\,1)$ 
and $v\in \widetilde{L}_1(-1,\,1)\cap  \widetilde{L}_2(-1,\,1).$  
The class $\widetilde{L}_1(-1,\,1)$ is distinguished by the condition $\int_{-1}^1 |v(t)| (1-t^2)^{-1}\,dt < \infty$.
Identity \eqref{0070} is valid under the same conditions on  $v$ and for $u \in \widetilde{H}^1$.
However, as in the case of the Fourier transform,  the conditions of applicability of relations \eqref{0060}, \eqref{0070} can be expanded significantly. In particular,  one can relax the requirements on  $v$, assuming $v$ to be  a distribution  and, if necessary, imposing more restrictive conditions on $u$. We will not go into details here.

\section{The weak solution of the Prandtl equation. Approach via the transform $\caP$\label{sec3}}

\subsection{``Another'' definition of the weak solution}
Now, we apply the transform $\caP$ to the study of problem \eqref{2.1.10}. 
As above, we start from the formal considerations. 
Using the boundary conditions  $u(-1)=u(1)=0$ and the identity {$1-x^2=1-xt+x(t-x),$} we have
\begin{equation}
-(1-x^2) \intop_{-1}^1 {u'(t)\over t-x}\,dt = \intop_{-1}^1  u'(t){1-xt\over x-t}\,dt. \label{00100}
\end{equation}
The form of the right-hand side allows us to apply the convolution formula  \eqref{0070} with $v(t) = t^{-1}$ in order to calculate the $\caP$-image of the function \eqref{00100}.
Taking \eqref{0090} into account, we obtain 
\begin{equation}
\caP\left[-(1-x^2){1 \over 2\pi}\intop_{-1}^1 {u'(t)\over t-x}\,dt\right](\xi)= \xi\,\coth\pi\xi\, U(\xi).\label{00105}
\end{equation} 
Multiplying  \eqref{2.1.10} by some function $\overline{g(x)}$, integrating over the interval, and using  \eqref{0050},  \eqref{00105}, we arrive at 
\begin{equation}
\intop_{-1}^1 V(x) u(x)\, \overline{g(x)} \,dx+\frac{1}{\pi}\intop_{-\infty}^\infty \xi\,\coth\pi\xi\, U(\xi)\,\overline{G(\xi)}\,d\xi=\intop_{-1}^1 f(x)\,\overline{g(x)}\,dx.
\label{00107}
\end{equation} 
The sesquilinear form in the left-hand side of \eqref{00107} is denoted by 
\begin{equation}
\label{010}
[u,\,g]_1:=\intop_{-1}^1 V(x) u(x)\, \overline{g(x)} \,dx+\frac{1}{\pi}\intop_{-\infty}^\infty \xi\,\coth\pi\xi\, U(\xi)\,\overline{G(\xi)}\,d\xi.
\end{equation}

The natural class to look for the weak solution is the space 
$\widetilde{H}^{1/2}(-1,1)$; see Subsection \ref{sec2.5}.
The form $[u,u]_1^{1/2}$ defines the norm in  $\widetilde{H}^{1/2}$ equivalent to the standard norm. Indeed, by the estimates
\begin{equation}
\label{**11}
 \frac{1}{\pi^2}  + \frac{2}{3} \xi^2 \le  \xi^2\,\coth^2\pi\xi \le \frac{1}{\pi^2} + \xi^2, \quad \xi \in \R,
\end{equation} 
$$
\intop_{-1}^1  V(x)|u(x)|^2 \,dx \le M \intop_{-1}^1  \frac{|u(x)|^2}{1- x^2}\,dx
= \frac{M}{\pi} \intop_{-\infty}^\infty |U(\xi)|^2 \, d\xi \le M \| u \|^2_{\widetilde{H}^{1/2}},\quad u \in \widetilde{H}^{1/2},
$$
we  see that 
\begin{equation*}
 \frac{1}{\pi}\| u \|^2_{\widetilde{H}^{1/2}} \le [u,u]_1 \le (M +1/2) \| u\|^2_{\widetilde{H}^{1/2}}, 
 \quad u\in \widetilde{H}^{1/2}.
\end{equation*}
Then the sesquilinear form $[u,g]_1$ given by  
  \eqref{010} can be taken as the inner product in $\widetilde{H}^{1/2}$.  
 By Proposition \ref{prop_plotno}, the set $C_0^\infty(-1,1)$ is dense in~$\widetilde{H}^{1/2}$.

 Relation \eqref{00107}  can be written as 
\begin{equation*}
[u,\,g]_1 =(f,\,g).
\end{equation*}
A natural class for $f$ is the space $(\widetilde{H}^{1/2})^*$ dual to 
$\widetilde{H}^{1/2}$ with respect to the pairing in $L_2(-1,1)$. 
(It is defined similarly to the space $(H^{1/2}_{00})^*$.)
The norm in $(\widetilde{H}^{1/2})^*$ is given by    
 \begin{equation*}
\|f\|_{(\widetilde{H}^{1/2})^*} = \sup_{0 \ne u \in \widetilde{H}^{1/2}}
 \frac{|(f,u)|}{\|u\|_{\widetilde{H}^{1/2}}}.
\end{equation*}

\begin{remark}
For $u,g \in C_0^\infty(-1,1)$ we have  $[u,g]_1 = [u,g]$, where $[u,g]$ is given by  
\eqref{2.1.30a}. Indeed, both expressions $[u,g]$ and $[u,g]_1$  are equal to the integral of the 
left-hand side of \eqref{2.1.10} multiplied by $\overline{g(x)}$. Using that $C_0^\infty(-1,1)$ is dense in 
  $H_{00}^{1/2}$, as well as in $\widetilde{H}^{1/2}$,  $[u,u]^{1/2}$ defines the norm in $H_{00}^{1/2}$ equivalent to 
  the standard one, and $[u,u]_1^{1/2}$ defines the norm in $\widetilde{H}^{1/2}$ equivalent to the standard one, we conclude that  $\widetilde{H}^{1/2} = H_{00}^{1/2}$
  and $[u,g]_1 = [u,g]$ for any $u,g \in  H_{00}^{1/2}$.
  Hence, the space $(\widetilde{H}^{1/2})^*$ coincides with $(H_{00}^{1/2})^*$.
\end{remark}

Now, on the basis of identity \eqref{00107}, we can give the definition of the weak solution of problem 
 \eqref{2.1.10}, which is independent of Definition~\ref{def1.1}.

\begin{definition}\label{def3.2}
Let $f \in (\widetilde{H}^{1/2})^*$. An element $u\in \widetilde{H}^{1/2}$ satisfying the integral identity \eqref{00107} for any  $g\in \widetilde{H}^{1/2}$ is called the weak solution of problem   \eqref{2.1.10}.
\end{definition}

The following theorem is equivalent to  Theorem \ref{th1.2} and is proved in a similar way 
(with the help of the Riesz theorem). 

\begin{theorem}\label{th3.3}
Suppose that  $V(x)$ satisfies \eqref{p}. Then for any 
$f \in  (\widetilde{H}^{1/2})^*$ there exists a unique weak solution $u\in \widetilde{H}^{1/2}$ of problem   \eqref{2.1.10}. The solution satisfies the following estimate\emph{:} 
\begin{equation}\label{***1}
\|u \|_{\widetilde{H}^{1/2}} \le   \pi \|f\|_{(\widetilde{H}^{1/2})^*}.
\end{equation}
\end{theorem}

\textbf{Example}.
Suppose that $f \in L_{2,r}(-1,1),$ see \eqref{H}. Then 
\begin{equation*} 
|(f,g)| \le  \|f\|_{L_{2,r}} \|g\|_{\widetilde{H}^{1/2}},\quad g \in \widetilde{H}^{1/2}, 
\end{equation*}
cf. \eqref{2.1.47}. Hence,  $L_{2,r} \subset (\widetilde{H}^{1/2})^*$ and 
$\|f\|_{(\widetilde{H}^{1/2})^*} \le \|f\|_{L_{2,r}}$.
Together with \eqref{***1}, this implies that 
\begin{equation*}
\|u \|_{\widetilde{H}^{1/2}} \le {\pi}  \|f\|_{L_{2,r}},
\end{equation*}
cf. \eqref{012}. In what follows, we will also need the estimate  
\begin{equation}
\label{013a}
\intop_{-1}^1  V(x) |u(x)|^2  \, dx \le   \frac{\pi}{4}  \|f\|^2_{L_{2,r}},
\end{equation}
which follows from the identity $[u,u]_1 = (f,u)$. Indeed, we have 
$$
 \begin{aligned}
   &\intop_{-1}^1 V(x) |u(x)|^2\,dx+\frac{1}{\pi}\intop_{-\infty}^\infty \xi\,\coth\pi\xi\, |U(\xi)|^2\,d\xi=\intop_{-1}^1 f(x)\,\overline{u(x)}\,dx
 \le \|f\|_{L_{2,r}} \| u\|_{\widetilde{L}_2}  =
 \\ 
 &= \|f\|_{L_{2,r}} \left( \frac{1}{\pi}\intop_{-\infty}^\infty |U(\xi)|^2 \, d\xi\right)^{1/2} \le \|f\|_{L_{2,r}} \left(\intop_{-\infty}^\infty  \xi\,\coth\pi\xi\, |U(\xi)|^2 \, d\xi\right)^{1/2} \le
 \\
 &\le \frac{\pi}{4}\|f\|_{L_{2,r}} ^2 + \frac{1}{\pi}\intop_{-\infty}^\infty \xi\,\coth\pi\xi\, |U(\xi)|^2\,d\xi.
 \end{aligned}
 $$
We have used the Parceval identity \eqref{0050}, the lower estimate \eqref{**11}, and the elementary inequality  $ab \le \alpha a^2 + \frac{1}{4\alpha} b^2$ (for positive numbers $a,b$ with arbitrary $\alpha >0$).

In the next subsection, we will show that for $f \in L_{2,r}$ the solution is more regular.

\subsection{Improvement of regularity of the solution}

In this subsection, it is assumed that  $f \in L_{2,r}(-1,1),$ i.~e.,  \eqref{H} is satisfied.
Obviously, the space $ L_{2,r}(-1,1)$ is dual to $\widetilde{L}_2(-1,1)$ with respect to the pairing in $L_2(-1,1)$. In other words,  $L_{2,r}(-1,1) =(\widetilde{H}^0)^*$ and 
$$
\|f\|_{(\widetilde{H}^0)^*}= \| f\|_{L_{2,r}}.
$$ 

The integral identity \eqref{00107} can be written as 
\begin{equation}
\frac{1}{\pi} \intop_{-\infty}^\infty \xi\,\coth\pi\xi\, U(\xi)\,\overline{G(\xi)}\,d\xi=\intop_{-1}^1 \left(f(x)- V(x)u(x)\right)\overline{g(x)}\,dx, \quad g \in \widetilde{H}^{1/2}.
\label{001073}
\end{equation} 
Denote 
$$
Q(\xi)=\caP\left[(1-x^2)\left(f(x)- V(x) u(x)\right)\right](\xi)=\intop_{-1}^1 \left(f(x)- V(x)u(x) \right)\left({1-x\over1+x}\right)^{i\xi}\,dx.
$$
It is easily seen that the function $(1-x^2)\left(f(x)- V(x) u(x) \right)$ belongs to  
$\widetilde{L}_2(-1,1)$, whence 
$Q \in L_2(\R)$. Indeed, by  \eqref{p}, \eqref{0050}, and \eqref{013a}, 
\begin{equation}
\begin{aligned}
& \frac{1}{\pi} \intop_{-\infty}^\infty |Q(\xi)|^2\,d\xi
=\intop_{-1}^1 (1-x^2)\left|f(x)- V(x)u(x)\right|^2\,dx\le
\\ 
&\le 2 \intop_{-1}^1 (1-x^2)|f(x)|^2\,dx+2
\intop_{-1}^1 (1-x^2) V(x) \cdot V(x) |u(x)|^2\,dx 
\le C_1  \intop_{-1}^1 (1-x^2)\,|f(x)|^2\,dx,
\label{001077}
\end{aligned}
\end{equation}
where $C_1 = 2+  \frac{\pi}{2} M$.

Now, using \eqref{0050}, we represent identity \eqref{001073} as 
\begin{equation}
 \intop_{-\infty}^\infty \xi\,\coth\pi\xi\, U(\xi)\,\overline{G(\xi)}\,d\xi=
\intop_{-\infty}^\infty Q(\xi)\,\overline{G(\xi)}\,d\xi,
\quad G \in \caP[ \widetilde{H}^{1/2}].
\label{001075}
\end{equation} 
For $N>0$, we define the following cut-off function 
$$
w(\xi,N)=\begin{cases} 1,&|\xi|<N\\ 0,& |\xi|> N +1\\ -\xi+N+1,& N <\xi< N+1 \\ \xi+N+1,& -N-1 <\xi<-N
\end{cases}
$$
and take the test function of the form 
\begin{equation}\label{dv5}
G_N(\xi)=\xi\,\coth\pi\xi\,U(\xi)\,w(\xi,N).
\end{equation}
Let us check that  $G_N \in \caP[ \widetilde{H}^{1/2}]$. Let $g_N = \caP^{-1}(G_N)$. Then from  \eqref{**11} it follows that 
\begin{equation}\label{*1}
\begin{aligned}
\|g_N\|_{ \widetilde{H}^{1/2}}^2 &=\frac{1}{\pi}  \intop_{-\infty}^\infty (1+ 4 \xi^2)^{1/2} |G_N(\xi)|^2 \, d\xi \le
\\
&\le \frac{1}{\pi}  \intop_{-N-1}^{N+1} (1+ 4 \xi^2)^{1/2}  \xi^2 \coth^2 \pi\xi \, |U(\xi)|^2 \, d\xi
\le C(N) \|u\|^2_{\widetilde{H}^{1/2}},
\end{aligned}
\end{equation}
where $C(N) = \frac{1}{\pi^2}+(N+1)^2$.
Thus, $g_N \in  \widetilde{H}^{1/2}$, and we can substitute $G_N \in \caP[ \widetilde{H}^{1/2} ]$ 
as a test function in identity \eqref{001075}. We obtain 
\begin{equation}\label{*4}
\!\intop_{-\infty}^\infty\! \xi^2\,\coth^2\pi\xi \,|U(\xi)|^2\,w(\xi, N)\,d\xi
=\intop_{-\infty}^\infty \xi \,\coth \pi\xi \, Q(\xi) \,w(\xi, N) \overline{U(\xi)}\,d\xi.
\end{equation} 

Next, by analogy with \eqref{*1}, it is easily seen that the function $\widetilde{G}_N(\xi)= Q(\xi) \,w(\xi, N)$
belongs to the class $\caP[\widetilde{H}^{1/2}]$. From identity \eqref{001075} with $G(\xi)=\widetilde{G}_N(\xi)$ it follows that  
\begin{equation}
 \intop_{-\infty}^\infty \xi \,\coth \pi\xi \, Q(\xi) \,w(\xi,N) \overline{U(\xi)}\,d\xi
=
\intop_{-\infty}^\infty | Q(\xi)|^2 \,w(\xi,N) \,d\xi.
\end{equation}
Together with \eqref{*4}, this yields 
\begin{equation}
 \!\intop_{-\infty}^\infty\! \xi^2\,\coth^2\pi\xi \,|U(\xi)|^2\,w(\xi,N)\,d\xi
=\intop_{-\infty}^\infty | Q(\xi)|^2 \,w(\xi,N) \,d\xi.
\end{equation} 
Consequently,
\begin{equation}
\!\intop_{-\infty}^\infty\! \xi^2\,\coth^2\pi\xi \,|U(\xi)|^2\,w(\xi,N)\,d\xi
\le \intop_{-\infty}^\infty | Q(\xi)|^2  \,d\xi.
\end{equation} 
Letting $N$ tend to infinity and applying  the Fatou theorem, we conclude that the integral 
$\int_{-\infty}^\infty \xi^2\,\coth^2\pi\xi \,|U(\xi)|^2\,d\xi$ converges and satisfies the estimate  
\begin{equation}\label{*6}
\!\intop_{-\infty}^\infty\! \xi^2\,\coth^2\pi\xi \,|U(\xi)|^2\,d\xi
\le \intop_{-\infty}^\infty | Q(\xi)|^2  \,d\xi.
\end{equation} 

 From \eqref{001077} and \eqref{*6} it follows that 
\begin{equation*}
\intop_{-\infty}^\infty\! \xi^2\,\coth^2\pi\xi \,|U(\xi)|^2\,d\xi
\le {\pi C_1} \intop_{-1}^1 (1-x^2) | f(x)|^2  \,dx.
\end{equation*}
Combining this with \eqref{**1aa} and \eqref{**11}, we obtain that $u \in \widetilde{H}^1$ and 
\begin{equation*}
\|u\|^2_{\widetilde{H}^1} \le C_2  \intop_{-1}^1 (1-x^2) | f(x)|^2  \,dx,
\quad C_2 =  {\pi^2 C_1} = {\pi^2} \left( 2 +  \frac{\pi}{2} M\right).
\end{equation*} 
By Proposition \ref{prop_2.1}, the solution is continuous on $[-1,1]$ and satisfies the boundary 
conditions $u(-1) = u(1)=0$. 
 Note that the solution $u \in \widetilde{H}^1$ satisfies identity \eqref{00107} for any test function 
 $g \in \widetilde{H}^0$.
 Moreover, for this solution the initial statement of the problem \eqref{2.1.10} can be used (the equation is satisfed almost everywhere, the boundary conditions are fulfilled, and therefore there is no need to replace the problem by the integral identity). Such solution is called the   \emph{strong solution}.
 
As a result, we arrive at the following theorem about regularity of the solution. 

\begin{theorem}\label{th3.4}
Suppose that  $V(x)$ satisfies  \eqref{p}.  Suppose that $f(x)$
is subject to condition  \eqref{H}.
 Then the weak solution $u(x)$ of problem  \eqref{2.1.10} belongs to the class $\widetilde{H}^{1}(-1,1)$
 and satisfies the estimate
\begin{equation*}
\intop_{-1}^1 \left( \frac{|u(x)|^2}{1-x^2} + (1-x^2) |u'(x)|^2 \right) dx
\le C_2 \intop_{-1}^1 (1-x^2) | f(x)|^2  \,dx.
\end{equation*}
The constant $C_2= {\pi^2} \left( 2 +  \frac{\pi}{2} M\right)$ depends only on the constant $M$ from condition \eqref{p}.
\end{theorem}

\subsection{Interpolation}

Let $R$ be the operator taking  $f$ into the solution of problem \eqref{2.1.10}. 
Theorems \ref{th3.3} and \ref{th3.4} show that the operator $R$ is continuous from  
$(\widetilde{H}^{1/2})^*$ to $\widetilde{H}^{1/2}$ and from $(\widetilde{H}^{0})^*$ to $\widetilde{H}^{1}$.
We have
\begin{align}
\label{interp1}
\| R \|_{(\widetilde{H}^{1/2})^* \to \widetilde{H}^{1/2}} \le {\pi},
\\
\label{interp2}
\| R \|_{(\widetilde{H}^{0})^* \to \widetilde{H}^{1}} \le \sqrt{C_2}.
\end{align}
Denote by $(\widetilde{H}^s)^*$ the space dual to $\widetilde{H}^{s}$ with respect to the pairing in 
$L_2(-1,1)$. Interpolating between  \eqref{interp1} and \eqref{interp2}, we obtain 
$$
\| R \|_{(\widetilde{H}^{1/2 - \theta})^* \to \widetilde{H}^{1/2 + \theta}} \le C_\theta=
 {\pi}^{1-2\theta} C_2^\theta, \quad 0\le \theta \le \frac{1}{2}.
$$
We arrive at the following result.

\begin{theorem}\label{th3.5}
Suppose that  $V(x)$ satisfies conditions \eqref{p}.  Let \hbox{$0\! \le\! \theta\! \le\! 1/2$} and
$f \in (\widetilde{H}^{1/2 - \theta})^*$. Then the weak solution $u(x)$ of problem  \eqref{2.1.10} belongs to the class $\widetilde{H}^{1/2 +\theta}(-1,1)$ and satisfies the estimate
\begin{equation*}
\| u\|_{\widetilde{H}^{1/2 +\theta}} \le C_\theta \| f\|_{(\widetilde{H}^{1/2 - \theta})^*}.
\end{equation*}
The constant $C_\theta$ depends only on the constant $M$ from condition \eqref{p} and on $\theta$.
\end{theorem}

Theorem \ref{th3.5} and Proposition \ref{prop_2.1} imply the following corollary.

 \begin{corollary}
 Under condition $f \in (\widetilde{H}^{1/2 - \theta})^*$, where $0 < \theta \le 1/2$,  the solution $u(x)$ of problem \eqref{2.1.10} is continuous on the closed interval $[-1,1]$ and satisfies conditions $u(-1) = u(1)=0$. 
 We have
 \begin{equation*}
 \| u \|_{C[-1,1]}  \le  C(1/2 + \theta) C_\theta \, \| f \|_{(\widetilde{H}^{1/2 - \theta})^*}.
\end{equation*}
 \end{corollary}
 
\subsection{Duality}
For completeness, we consider the so called ``very weak'' solution from the class 
$\widetilde{L}_2(-1,1)$, assuming that $f \in (\widetilde{H}^1)^*$.

\begin{definition}\label{def3.7}
Let $f \in (\widetilde{H}^1)^*$. A function $u \in \widetilde{L}_2(-1,1)$ satisfying identity \eqref{00107} for any test function $g \in \widetilde{H}^1(-1,1)$ is called the very weak solution of problem \eqref{2.1.10}. 
\end{definition}

 \begin{theorem}\label{th3.8}
 Suppose that  $V(x)$ satisfies conditions \eqref{p}.
 For any $f \in (\widetilde{H}^1)^*$ there exists a unique very weak solution  
 $u \in \widetilde{L}_2(-1,1)$ of problem \eqref{2.1.10}. We have 
\begin{equation}
\label{dv0}
\| u \|_{\widetilde{L}_2}  \le \sqrt{C_2} \| f \|_{(\widetilde{H}^1)^*}.
\end{equation}
 \end{theorem} 
 
 \begin{proof}
 Let $R: (\widetilde{H}^{0})^* \to \widetilde{H}^{1}$ be the resolving operator from Theorem \ref{th3.4}. 
Then the continuous adjoint operator $R^*: (\widetilde{H}^{1})^*\to \widetilde{H}^{0}$ is defined correctly
by the relation 
\begin{equation}
\label{dv1}
(R g, f) = (g, R^* f), \quad g \in (\widetilde{H}^{0})^*,\ f \in (\widetilde{H}^{1})^*.
\end{equation}

Fix $f \in (\widetilde{H}^{1})^*.$ Let us check that $u = R^* f \in   \widetilde{H}^{0}$ is the very weak solution of problem \eqref{2.1.10}.

Let $g \in (\widetilde{H}^{0})^*$. Then $v := Rg \in \widetilde{H}^1$ satisfies the identity 
$$
[v, u]_1 = (g,u), \quad \forall u \in \widetilde{H}^0. 
$$ 
Here the form $[v,u]_1$ (see \eqref{010}) is extended  
to the pairs $v \in \widetilde{H}^1$ and $u \in \widetilde{H}^0$. In other words,
\begin{equation}
\label{dv2}
[Rg, u]_1 = (g,u), \quad \forall g \in (\widetilde{H}^{0})^*,\ u \in  \widetilde{H}^{0}.
\end{equation}

Substituting the function $u = R^* f \in   \widetilde{H}^{0}$ in \eqref{dv2},  we obtain 
\begin{equation}
\label{dv3}
[Rg, u]_1 = (g, R^* f) = (Rg, f), \quad \forall g \in (\widetilde{H}^{0})^*.
\end{equation}
We have used  \eqref{dv1}.
Note that,  if $g$ runs over $(\widetilde{H}^{0})^*$, then $v =Rg$ runs over $\widetilde{H}^1$.
Therefore, identity \eqref{dv3} can be written as 
$$
[v, u]_1 =  (v, f), \quad \forall v \in \widetilde{H}^1.
$$
 Since the  form \eqref{010} is Hermitian, this  is equivalent to the identity  
$$
[u, v]_1 =  (f, v), \quad \forall v \in \widetilde{H}^1.
$$
This means that $u = R^* f \in   \widetilde{H}^{0}$ is the very weak solution of problem \eqref{2.1.10}.

To prove uniqueness, assume that $u \in \widetilde{L}_2(-1,1)$ is the very weak solution 
of problem \eqref{2.1.10} with $f=0$. By analogy with \eqref{001073}--\eqref{001075}, we represent the identity 
for $u$ in the form 
\begin{equation}
 \intop_{-\infty}^\infty \xi\,\coth\pi\xi\, U(\xi)\,\overline{G(\xi)}\,d\xi=
\intop_{-\infty}^\infty Q(\xi)\,\overline{G(\xi)}\,d\xi,
\quad G \in \caP[ \widetilde{H}^{1}].
\label{dv4}
\end{equation} 
We have $Q \in L_2(\R)$ and $\|Q\|_{L_2(\R)} \le \sqrt{\pi}M \|u\|_{\widetilde{L}_2}$.    
Substituting the test function $G_N(\xi)$ of the form \eqref{dv5}
(which belongs to $\caP[ \widetilde{H}^{1}]$) in \eqref{dv4}, by analogy with  \eqref{*4}--\eqref{*6},
we prove that $u \in \widetilde{H}^1$.
Now, Theorem \ref{th3.3} (the uniqueness part) implies that $u=0$. 

Estimate \eqref{dv0} follows from  \eqref{interp2} and the equality 
$\| R^* \|_{(\widetilde{H}^{1})^*\to \widetilde{H}^{0}} = \|R\|_{(\widetilde{H}^{0})^* \to \widetilde{H}^{1}}$.
\end{proof}

From the uniqueness of the very weak solution  
it follows that the resolving operator $R^* : (\widetilde{H}^{1})^*\to \widetilde{H}^{0}$
is an extension of the operator $R: (\widetilde{H}^{1/2})^*\to \widetilde{H}^{1/2}$.
Keeping the same notation $R$ for the extended operator, we rewrite  \eqref{dv0} as 
\begin{equation}
\| R \|_{(\widetilde{H}^{1})^*\to \widetilde{H}^{0}} \le \sqrt{C_2}.
\label{dv10}
\end{equation} 
Interpolating between \eqref{dv10} and \eqref{interp2}, we obtain 
$$
\| R \|_{(\widetilde{H}^{1 - s})^* \to \widetilde{H}^{s}} \le  \sqrt{C_2}, \quad 0\le s \le 1.
$$

We arrive at the following final result, which combines the statements of all previous theorems. 

\begin{theorem}\label{th3.9}
Suppose that  $V(x)$ satisfies conditions \eqref{p}.
 Let $0 \le s \le 1$. For any $f \in (\widetilde{H}^{1-s})^*$ there exists a unique solution 
 $u \in \widetilde{H}^{s}(-1,1)$ of problem \eqref{2.1.10}. We have 
\begin{equation*}
\| u \|_{\widetilde{H}^{s}}  \le \sqrt{C_2} \| f \|_{(\widetilde{H}^{1-s})^*}.
\end{equation*}
\end{theorem}

In Theorem \ref{th3.9}, for $0 \le s < 1/2$ the solution is understood as the very weak solution in the sense of Definition \ref{def3.7},
 for $1/2 \le s < 1$ the solution is understood as the weak solution in the sense of Definition \ref{def3.2}, and for $s=1$ the solution is understood as the strong solution.

\end{document}